\documentclass{article}
\usepackage{inputenc}
\usepackage{amsmath,amsthm,amssymb,amsfonts}
\usepackage{url}
\usepackage{color}
\usepackage{caption,subcaption}
\usepackage{textcomp}
\usepackage{siunitx}
\usepackage{multirow}
\usepackage{float}
\usepackage{pdflscape}
\usepackage{afterpage}
\usepackage{cleveref}
\usepackage[linesnumbered, ruled]{algorithm2e}
\usepackage{pseudocode}
\usepackage[toc,page,titletoc,title]{appendix}

\newtheorem{theorem}{Theorem}[section]
\theoremstyle{plain}

\newtheorem{proposition}[theorem]{Proposition}

\newtheorem{observation}[theorem]{Observation}

\newtheorem{ex}{Example}

\author{Gerold J\"ager \\ gerold.jager@umu.se 
	\and Klas Markstr\"om \\ klas.markstrom@umu.se 
	\and Denys Shcherbak \\ denys.shcherbak@umu.se
	\and Lars-Daniel \"Ohman \\ lars-daniel.ohman@umu.se 
	}

\begin{document}

\title{Triples of Orthogonal Latin and Youden Rectangles For Small Orders}

\maketitle

\begin{abstract}
	We have performed a complete enumeration of non-isotopic triples of
	mutually orthogonal 
	$k\times n$ Latin rectangles for $k\leq n \leq 7$. Here we will present 
	a census of such triples, classified by various properties, including 
	the order of the autotopism group of the triple. As part of this we have 
	also achieved the first enumeration of pairwise orthogonal triples of Youden 
	rectangles.
	We have also studied orthogonal triples of $k \times 8$ rectangles which 
	are formed by extending	mutually orthogonal triples with non-trivial 
	autotopisms one row at a time, and requiring that the autotopism group
	is non-trivial in each step. This class includes a triple coming from 
	the projective plane of order 8. Here we find a remarkably symmetrical pair of 
	triples of $4 \times 8$ rectangles, formed by juxtaposing two selected copies 
	of complete sets of MOLS of order 4.
\end{abstract}

\section{Introduction}

Sets of mutually orthogonal Latin squares (MOLS) are well studied objects in 
both pure combinatorics and statistical design theory, see e.g.\ Chapter 10 of  
\cite{HK}, with a history going at least all the way back to Euler's 36 officers 
problem and the conjectures he made based on his studies of that problem. In the 
present paper we will study a variation on this topic, namely sets of mutually 
orthogonal Latin rectangles. For $k\leq n $, a Latin rectangle is a $k \times n$ 
array using $n$ 
different symbols, such that each row is a permutation of the symbols, and all 
symbols in a column are distinct. In particular, an $n \times n$ Latin rectangle 
is a Latin square. The concept of orthogonality in Latin squares can be naturally 
extended to rectangles by saying that the rectangles $R$ and $L$ are orthogonal 
if all the ordered pairs of symbols $(R_{ij},L_{ij})$ are distinct. More 
specifically, the aim has been to classify all non-isotopic triples of MOLS and 
mutually orthogonal Latin rectangles (MOLR) for as large $n$ as possible, which 
turns out to be $n=7$.

One reason for choosing to work with triples is a fundamental question in the 
theory of mutually orthogonal Latin squares: How large can a set of MOLS be for 
a given $n$?  The  maximum size of such a set is here denoted by $N(n)$. 
Thanks to the now classical results by Bose, Tarry and others \cite{constr, 
disprove, Ta} we know that for $n\neq 2,6$ there exist MOLS, i.e. 
$N(n)\geq 2$ for $n\neq 2,6$. It was also proven by Chowla, {Erd\H{os}} and  
Straus \cite{CES} that there exist constants $C, \alpha>0$ such that $N(n)\geq C 
n^{\alpha}$ for large enough $n$.  This result has been improved in various 
ways over the years and recently  Barber, K\"uhn, Lo, Osthus and Taylor \cite{BKLO} 
proved that one can prescribe a substantial number of entries in $r$ partial 
Latin squares and still be able to extend these partial Latin squares to a set of 
$r$ MOLS, where $r$ grows with $n$. 

For small $n$, a sequence of results, culminating with Todorov \cite{To}, lead 
to the result that $N(n)\geq 3$ unless $n=2, 3, 6$ and possibly $n=10$.   The 
open case  $n=10$ is of special interest due to its connection to the existence 
of finite projective planes; a finite projective plane of order $n$ is 
equivalent to a set of $n-1$ MOLS. It is believed 
that a projective plane only exists for orders which are prime powers and the first 
open case for the conjecture was for a long time $n=10$. This case was settled 
in 1989 \cite{LTS} by a combination of methods from coding theory and an 
extensive computer search, an effort well surveyed in \cite{Lam}. That result in 
fact shows that $N(10)\leq 6$, but the result would of course itself follow if 
in fact $N(10)=2$. A number of authors have  attempted to construct triples of 
MOLS for $n=10$ and the latest large scale such effort \cite{MMM} proved that if 
such a triple exists, then each of the three Latin squares involved must have a trivial 
symmetry group. For Latin rectangles it is however trivial that triples of 
orthogonal $1\times n$ rectangles exist, and it is not hard to show that triples of 
mutually orthogonal $2 \times n$ rectangles exist if $n \geq 3$, as noted in \cite{Agg}. 
One reason for studying triples of MOLR is therefore that if there is no triple of 
MOLS of order $10$, there must be a largest $k<10$ such that a triple of mutually 
orthogonal $k\times 10$ rectangles does exist.

In a more general perspective, with the double aim of finding rectangular 
orthogonal arrays for design purposes and getting a better understanding of the 
existence of MOLS we have generated and classified all orthogonal triples of 
MOLR up to $n=7$. For each 
fixed $n$, the number of non-isotopic $k \times n$ triples turns out to be a 
unimodal function of $k$, and even for $n=6$, where no pair of orthogonal 
Latin squares exists, there are several examples of triples of MOLR with $k=5$.    
For each value of 
$n$ we have also counted the number of non-isotopic maximal triples, i.e., 
triples which 
cannot be extended by one more row. These results are described in Subsection 
\ref{ssec_triples}. Further, as described in Subsection \ref{ssec_auto}, we have computed 
the autotopism groups of the orthogonal triples. Here we see a clear change in 
behaviour when $n$ increases. For $n \leq 5$ the autotopism groups are always 
non-trivial, but at $n=6$ we see the first examples with no autotopisms, and 
for $n=7$ this is the dominant case.

The paper is structured as follows. In Section~\ref{sec_not} we give notation and 
formal definitions. In Section~\ref{sec_problem} we state the questions guiding our 
investigation, describe briefly the algorithm 
used to find all triples of MOLR and give some practical information regarding the 
computer calculations. In Subsection~\ref{ssec_triples} we 
present the data our computer search resulted in, in particular the number of 
non-isotopic triples of MOLR for each order. By means of analysis of the data produced, 
we discuss the extendability of the triples. In Subsection~\ref{ssec_youden} we 
discuss which of the triples found satisfy the stronger condition of being Youden 
rectangles and in Subsection~\ref{ssec_auto} we discuss the autotopism groups of 
the triples. Up to order $n=7$ our enumeration is complete, and in 
Section~\ref{sec_larger} we discuss our enumeration of triples of order 
$n=8$ with the added condition that they sequentially for each $k$ have non-trivial 
autotopism groups.

\section{Notation and Definitions}
\label{sec_not}

A \emph{Latin square} is an $n\times n$ matrix with cells filled by $n$ symbols 
such that each row and each column contains a specific symbol exactly once. 
For $ k \leq n $ a matrix with $k$ rows and $n$ columns whose cells are filled by $n$ symbols such 
that each row contains each symbol exactly once and each column contains each 
symbol at most once is called a $k\times n$ \emph{Latin rectangle}. 
In the following we use as symbol set $\{0,1,\dots,n-1\}$.
We denote the $t$-th row of a $k\times n$ rectangle $A$ by $A_t$ for 
$t\in\{1,2,\dots, k\}$. 

Let $A=(\alpha_{i,j})_{1\leq i,j,\leq n}$ and $B=(\beta_{i,j})_{1\leq i,j,\leq 
n}$ be Latin squares. We say that $A$ and $B$ are \emph{mutually orthogonal Latin 
squares} if the set of ordered pairs $\{(\alpha_{i,j}, \beta _{i,j})\, |\, 
i,j \in\{1,2\dots n\}\}$ contains all possible ordered pairs, or, in other words, if 
each ordered pair $(\alpha_{i,j}, \beta _{i,j})$ appears exactly once. A set of 
Latin squares of order $n$ is called a set of \emph{Mutually Orthogonal Latin 
Squares (MOLS)} if each square is orthogonal to every other square in the set. 
Similarly, we can extend the orthogonality condition to Latin rectangles. We 
say that Latin rectangles of size $k\times n$ $A=(\alpha_{i,j})_{ 1\leq i\leq 
k,\, 1\leq j\leq n}$ and $B=(\beta_{i,j})_{ 1\leq i\leq k,\, 1\leq j\leq n}$ are 
orthogonal if each ordered pair $(\alpha_{i,j}, \beta _{i,j})$  appears at most 
once. Also, a set of pairwise orthogonal Latin rectangles is called a set of
\textit{Mutually Orthogonal Latin Rectangles (MOLR)}.

There are many different notions of `equivalence' of Latin squares, Latin 
rectangles, mutually orthogonal Latin squares, and mutually orthogonal Latin 
rectangles (see, for example, \cite{LaywineMullen} and \cite{WanlessEgan}).
In the present paper, we use the equivalence notion \emph{isotopism}: Two 
triples $T_1$ and $T_2$ of MOLR are said to be \emph{isotopic} if $T_2$ can be 
gotten from $T_1$ by permuting the order of the three rectangles, by permuting 
rows (jointly in all three rectangles in $T_1$), permuting columns (jointly in 
all three rectangles in $T_1$), and by permuting symbols (separately in each of 
the three rectangles in $T_1$). It should be noted that all these permutations
will preserve the orthogonality condition, and that there are no further
obvious transformations that will.

Considering that our main focus in the present paper is Latin rectangles, 
we have chosen not to use a stronger concept of `equivalence' sometimes used for
Latin \emph{squares}, where interchanges of the roles of rows, columns, and symbols 
are also allowed. In a Latin rectangle, rows, columns and symbols do not play the 
same role, so such interchanges would not result in Latin rectangles. In the special
case of $k=n$, which we also study, applying this stronger notion of `equivalence' 
is meaningful, and would result in so-called \emph{main classes}, or \emph{species}, 
of MOLS. This enumeration, however, has already been done in \cite{WanlessEgan} for 
all the $n$ included in our study, and so we do not repeat it here.

We will also be interested in \emph{how many} triples of MOLR are isotopic, which motivates the 
following slightly different presentation and additional terminology.
Let $(A,B,C)$ be a triple of MOLR of size $k\times n$. The following group of 
isotopisms acts on the set of triples of MOLR: $G_{n,k}=S_3\times S_k\times 
S_n \times [S_n \times S_n \times S_n]$, where $S_3$ corresponds to a permutation 
of the rectangles, $S_k$ corresponds to a permutation of the rows, $S_n$ corresponds 
to a permutation of the columns, and each of the last three $S_n$ correspond to a 
permutation of the symbols in a single rectangle. Two triples  $(A,B,C)$ and 
$(A^\prime,B^\prime,C^\prime)$ of MOLR of size $k\times n$ are isotopic if 
there exists a $g\in G_{n,k}$ such that  
$g(A,B,C)=(A^\prime,B^\prime,C^\prime)$. The \emph{autotopism group} of a 
triple $(A,B,C)$ is defined as  $Aut(A,B,C):=\{g\in G_{n,k}\, |\, 
g(A,B,C)=(A,B,C)\}$.
 
Since each row of a Latin square or a Latin rectangle can be seen as a 
permutation, sometimes we refer to rows as permutations.
We say that a permutation $\sigma\in S_n$, where $S_n$ denotes the symmetric group 
on $n$ elements, is \textit{lexicographically} smaller 
than a permutation $\pi\in S_n$ and write $\sigma<\pi$, if $\sigma(i)<\pi(i)$, where 
$i$ is the first position with $\sigma(i)\neq\pi(i)$. 
The lexicographical comparison can be extended to triples of MOLR as follows.
Let $(A,B,C)$ and $(A^\prime,B^\prime,C^\prime)$ be triples of $k\times n$ MOLR. 
Furthermore, let $A$, $B$ and $C$ consist of permutations $\sigma_{i,s}$, and 
$A^\prime$, $B^\prime$, $C^\prime$ consist of permutations 
$\pi_{i,s}$, respectively, where $\sigma_{i,s},\pi_{i,s}\in 
S_n$ for $i\in\{1,2\dots k\}$ and $s\in\{1,2,3\}$. The sequence of triples of 
permutations $\{(\sigma_{1,1},\sigma_{1,2},\sigma_{1,3}), 
(\sigma_{2,1},\sigma_{2,2},\sigma_{2,3}),\dots,(\sigma_{k,1},\sigma_{k,2},\sigma_{k,3})\}$ 
describes $(A,B,C)$ and $\{(\pi_{1,1},\pi_{1,2},\pi_{1,3}), 
(\pi_{2,1},\pi_{2,2},\pi_{2,3}),\dots,(\pi_{k,1},\pi_{k,2},\pi_{k,3})\}$ 
describes  $(A^\prime,B^\prime,C^\prime)$.
We say that a triple of rectangles $(A,B,C)$ is lexicographically smaller than 
$(A^\prime,B^\prime,C^\prime)$ if $\sigma_{i,s}<\pi_{i,s}$, where $i,s$ are the 
first indices in the sequences with $\sigma_{i,s}\neq\pi_{i,s}$. 
In other words, we compare triples of MOLR by rows.

We call a triple of MOLR \textit{normalized} if it satisfies the 
following conditions:

\begin{itemize}
	\item[(S1)] (Ordering among columns) 
		The first row of each rectangle is the identity permutation.
	\item[(S2)] (Ordering among rectangles) 
		The second row of the first rectangle is lexicographically larger 
		than the second row of the second rectangle, 
		and the second row of the second rectangle is larger 
		than the second row of the third rectangle. 
		In other words, if $a_1,a_2,a_3$ are symbols on the 
		positions (2,1) in the ordered triple, then it holds that $a_1>a_2>a_3$.
		Note that $a_1$, $a_2$, $a_3$ are pairwise 
		distinct, since all the ordered pairs with 
		the same symbol in each position occur in the first row. 
	\item[(S3)] (Ordering among rows)
		The second row in the first rectangle is lexicographically larger 
		than the third one, the third one is larger than the fourth one, and so on.
\end{itemize}

In Figure~\ref{order} we give an example of a normalized triple of MOLS.
\begin{figure}[H]
\begin{tabular}{ | c | c | c | c|}
\hline
0 & 1 & 2 & 3\\ \hline
3 & 2 & 1 & 0\\ \hline
2 & 3 & 0 & 1\\ \hline
1 & 0 & 3 & 2\\ \hline
\end{tabular}
\hfill
\begin{tabular}{ | c | c | c | c|}
\hline
0 & 1 & 2 & 3\\ \hline
2 & 3 & 0 & 1\\ \hline
1 & 0 & 3 & 2\\ \hline
3 & 2 & 1 & 0\\ \hline
\end{tabular}
\hfill
\begin{tabular}{ | c | c | c | c|}
\hline
0 & 1 & 2 & 3\\ \hline
1 & 0 & 3 & 2\\ \hline
3 & 2 & 1 & 0\\ \hline
2 & 3 & 0 & 1\\ \hline
\end{tabular}
\caption{A normalized triple of MOLS of order 4.}
\label{order}
\end{figure}

Finally, as our computations proceed by adding consecutive rows to triples of
Latin rectangles, we will have use for the following term: 
An \textit{extension} of size $k\times n$ is a triple of MOLR which results 
from a triple of MOLR of size $(k-1)\times n$ by adding one more row.

\section{Generating data}
\label{sec_problem}

\subsection{Guiding questions}

 Our approach 
is complete enumeration by computer for as large parameters as possible, and unless 
otherwise stated, we save all generated data. In particular, we do not only 
record the number of triples of MOLR found, but with some exceptions noted below, 
we save the triples of MOLR themselves.

With some size exceptions due to size restrictions, the data generated is available 
for download at \cite{Web}. Further details are given there.

The following questions serve as guides for which data to generate.

\begin{itemize}
\item[(Q1)] How many normalized $k\times n$ triples of MOLR are there?
\item[(Q2)] How many isotopism classes of $k\times n$ triples of MOLR are there?
\item[(Q3)] How many non-isotopic $k\times n$ triples of MOLR are maximal, i.e\@. cannot be extended by one more row?
\item[(Q4)] Are there sets of triples of MOLR that satisfy some stronger regularity 
conditions? In particular, are any of the triples of MOLR in fact triples of Youden rectangles?
\item[(Q5)] Which order does the autotopism group of each $k\times n$ triple 
of MOLR have?
\end{itemize}

We note that it is clear from the definitions that the number of non-isotopic 
triples of MOLR will be less than the number of normalized ones. 

\subsection{Algorithms}

We employ two main algorithms. 

\begin{enumerate}
	\item {\bf Extension Finding:} 
	This first algorithm finds $k\times n$ triples of MOLR as extensions of 
	all non-isotopic $(k-1)\times n$ triples of MOLR 
	and counts the number of instances which are maximal.
	\item {\bf Isotopism Rejection:} 
	This algorithm applies an isotopism rejection for $k\times n$ triples of 
	MOLR, keeping one representative of each isotopism class. 
	This algorithm also keeps track of the order of the autotopism groups.
\end{enumerate}

Pseudocode and more detailed explanations of the algorithms are given in 
Appendix~\ref{algo1} and Appendix~\ref{algo2}, respectively.

\subsection{Implementation and Execution}

We use Algorithms \ref{algo1_ext} and \ref{algo2_iso} in the following 
two computations.
First, we find all normalized $k\times n$  triples of MOLR. More precisely, we 
count the number of normalized $k\times n$ triples of MOLR, the number of 
isotopism classes and the number of $k\times n$  triples of MOLR which are 
maximal (thus addressing questions (Q1), (Q2) and (Q3)). 
Second, we classify the triples of MOLR according to the order of their autotopism 
group (thus addressing question (Q5)).
Checking the Youden property (question (Q4)) was done by a separate, simple program.

We have implemented the algorithms in C++. Each classification was done for  
$n=4, 5,6,7$ and $k\in\{2,3,\dots,n\}$. The computation for $n=4,5,6$ was easily done 
on a standard desktop. For $n \geq 7$ both the computational effort and the disc 
requirements were significantly larger. We therefore parallelized 
Algorithms~\ref{algo1_ext} and \ref{algo2_iso} and ran the experiments on the parallel 
machine Kebnekaise from the High Performance Computing Center North (HPC2N). 

The running time of the first computation on a 
standard desktop is less than one minute for all $n=4,5,6$.
For $n=7$ the number of triples of MOLR is greater than 1.4 billion, with more 
than 400 million isotopism classes, and the running time of both the programs 
(generation and isotopism reduction) is almost 200 core hours.
The running time for the computation of the autotopism 
groups is similar. It can be done in one minute for all sizes except $3\times 7$ 
and $4\times 7$, where the program requires about 80 core hours.

As will be described in Section~\ref{sec_larger}, 
the parallelized version of the algorithms were also used to study the 
number of $3\times 8$ triples of MOLR.

\section{Results and Analysis}
\label{sec_results}

We now turn to the results and analysis of our computational work. 

\subsection{The Number of Triples of MOLR}
\label{ssec_triples}

Our first result is an enumeration of triples of MOLR. 
Table~\ref{Result_table_1} lists the number of normalized  $k\times n$  triples, 
the number of isotopism classes and how many of the non-isotopic cases 
are maximal, i.e., cannot be extended by one more row.

In appendix~\ref{appendix:MOLS4}, the unique (up to isotopism) $4 \times 4$ triple 
is given, 
in appendix~\ref{appendix:MOLS5}, the unique (up to isotopism) $5 \times 5$ triple is given,
in appendix~\ref{appendix:MOLR6}, the 7 non-isotopic $5 \times 6$ triples are given,
and in appendix~\ref{appendix:MOLS7}, the 4 non-isotopic $7 \times 7$ triples are given.

\begin{table}[H]
\begin{center}
\begin{tabular}{|r|r|r|r|}
\hline
\multirow{2}{*}{Size} &\multirow{2}{*}{\shortstack{\#Normalized} } &\multicolumn{2}{c|}{\#Non-isotopic} \\
\cline{3-4} 
 &   & \multicolumn{1}{c|}{\#Total}  & \multicolumn{1}{c|}{\#Maximal}  \\ \hline
$2\times 4$ & 4 & 2 & 1  \\ \hline
$3\times 4$ & 2 & 1 & 0  \\ \hline
$4\times 4$ & 1 & 1 & --   \\ \hline \hline
$2\times 5$ & 224 & 4 & 3   \\ \hline
$3\times 5$ & 3   & 1 & 0   \\ \hline
$4\times 5$ & 2   & 1 & 0   \\ \hline
$5\times 5$ & 1   & 1 & --     \\ \hline \hline
$2\times 6$ & \num{65520} & 103 & 0    \\ \hline
$3\times 6$ & \num{16767} & 2\,572 & 1\,800 \\ \hline
$4\times 6$ & 2\,005 & 513 & 493  \\ \hline
$5\times 6$ & 31 & 7 & 7    \\ \hline
$6\times 6$ & 0 & 0 & --    \\ \hline \hline
$2\times 7$ & \num{25864320} & 2\,858 & 0    \\ \hline
$3\times 7$ & \num{200127181} & \num{65883453} & \num{30025}  \\ \hline
$4\times 7$ & \num{1292959311} & \num{323112477} & \num{322850101}   \\ \hline
$5\times 7$ & \num{273190} & \num{55545} &\num{55508}    \\ \hline 
$6\times 7$ & 42 & 16 & 12     \\ \hline 
$7\times 7$ & 4 & 4 &  --   \\ \hline 
\end{tabular}
\end{center}
\caption{The number of orthogonal $k\times n$ triples.}
\label{Result_table_1}
\end{table}

From Table~\ref{Result_table_1} we observe that the behavior of the number of 
non-isotopic triples and the total number of triples is similar. The number 
of non-isotopic triples increases with $k$ up to 
around $k\approx \frac{n}{2}$ and the increase is very fast, 
but in the steps after the peak the numbers rapidly decrease.  
There are at least two reasons for this decrease of the number of 
triples of MOLR. First, the orthogonality condition conflicts with the Latin 
rectangle condition. As a result orthogonality fails on some positions, and 
this becomes more frequent when $k$ increases. Second, the number of 
non-isotopic triples decreases because of isotopic extensions, i.e., 
non-isotopic triples sometimes produce isotopic extensions.

In Example \ref{ex1} we show a maximal triple of MOLR. 
This example is interesting since $n=5$ is the largest order where we find 
maximal triples with just two rows. 
 
\begin{ex}\label{ex1}
A maximal $2\times 5$ triple of MOLR: 
\smallskip

\begin{tabular}{ | c | c | c | c|c|}
\multicolumn{5}{c}{$A$}\\
\hline
$0$ & $1$ & $2$ & $3$ & $4$ \\ \hline
$4$ & $3$ & $1$ & $2$ & $0$ \\ \hline
\end{tabular}
\hfill
\begin{tabular}{ | c | c | c | c|c|}
\multicolumn{5}{c}{$B$}\\
\hline
$0$ & $1$ & $2$ & $3$ & $4$ \\ \hline
$3$ & $4$ & $0$ & $1$ & $2$ \\ \hline
\end{tabular}
\hfill
\begin{tabular}{ | c | c | c | c|c|}
\multicolumn{5}{c}{$C$}\\
\hline
$0$ & $1$ & $2$ & $3$ & $4$ \\ \hline
$2$ & $0$ & $3$ & $4$ & $1$ \\ \hline
\end{tabular} \\[1em]

Each of these rectangles can be extended to Latin rectangles with 3 rows by 12, 
13, 13 permutations, respectively, but the orthogonality condition cannot be 
satisfied.
\end{ex}

In general, the number of rows in a maximal triple must grow with $n$, as the 
following proposition shows. In particular, it follows from this result that 
there are no maximal triples of MOLR of order $2 \times 6$ or $2 \times 7$.

\begin{proposition}
	A triple of $k\times n$ MOLR is not maximal if $k\leq \frac{n-1}{3}$.
\end{proposition}
\begin{proof}
	Assume that the triple consists of the three rectangles $R_1, R_2, R_3$. We 
	want to show that under the condition on $k$  we can extend each rectangle by 
	one row, while preserving orthogonality.  
	We define bipartite graphs $G_i$, one for each 
	rectangle, which has $n$ vertices which correspond to the columns of $R_i$ 
	and $n$ vertices which 
	correspond to the  symbols of $R_i$. The edges of the graph are given as follows:
	We start out with all edges with one 
	endpoint in each vertex class and then delete every edge $(c,s)$  
	where the column $c$ contains the symbol $s$. Here each row of $R_i$  
	leads to the deletion of a perfect matching, so $R_i$ is a regular graph 
	with vertex degree $n-k$.
	
	Now, since $G_1$ is regular and bipartite, it has a perfect matching $M_1$, 
	and we will use the corresponding assignment of symbols as the new row $r_1$ 
	in $R_1$. The new row assigns symbols to each position, and for each symbol 
	there are now $k$ symbols which are forbidden in that position in the 
	other two rectangles because of the orthogonality condition.  
	These are the symbols which appear in symbol pairs in earlier rows.  
	
	We delete those edges from $G_2$ and $G_3$ to get the graphs 
	$G'_2$ and $G'_3$. Each symbol appears in the same number of pairs so the 
	edges deleted due to orthogonality induce a $k$-regular graph, this means 
	that $G'_2$ and $G'_3$ have degree at least $n-2k$ and still satisfy Hall's 
	condition if this is at least 1, and so we can find a perfect matching in 
	$G'_2$ which gives us a valid new row $r_2$ for $R_2$.
	
	We can now repeat this for $G_3$ as long as $n-3k \geq 1$, which is 
	equivalent to our assumption $k\leq \frac{n-1}{3}$.
\end{proof}

It is clear that the proof idea in this proposition can be extended to give 
a similar result for extendability of $m$-tuples of MOLR, but that is not of
interest in the present paper.

Despite the fact that some triples of rectangles cannot be extended by one extra 
row because of orthogonality issues, sometimes it is possible to partially 
extend them, in other words, to simultaneously fill some position in the next 
row of each rectangle while satisfying the Latin rectangle condition and the 
orthogonality condition. However, there are examples where we cannot 
simultaneously fill any position in the next row, as demonstrated in Example \ref{ex2}.

\begin{ex}\label{ex2}
Consider the following $5\times 6$ triple of MOLR and the unique extensions of each 
rectangle to a Latin square:

\begin{table}[H]

\begin{tabular}{ | c | c | c | c|c|c|}
\multicolumn{6}{c}{$A$}\\
\hline
0 & 1 & 2 & 3 & 4 &5\\ \hline
5 & 4 & 3 & 2 & 1 &0\\ \hline
4 & 5 & 1 & 0 & 3 &2\\ \hline
3 & 2 & 5 & 4 & 0 &1\\ \hline
2 & 0 & 4 & 1 & 5 &3\\ \hline \hline
\textbf{1} &\textbf{3} & \textbf{0} & \textbf{5} & \textbf{2} &\textbf{4}\\ \hline
\end{tabular}
\hfill
\begin{tabular}{ | c | c | c | c|c|c|}
\multicolumn{6}{c}{$B$}\\
\hline
0 & 1 & 2 & 3 & 4 & 5\\ \hline
4 & 5 & 1 & 0 & 3 &2\\ \hline
3 & 2 & 5 & 4 & 0 &1\\ \hline
2 & 4 & 3 & 1 & 5 &0\\ \hline
5 & 3 & 0 & 2 & 1 &4\\ \hline \hline
\textbf{1} & \textbf{0} & \textbf{4} &\textbf{5} & \textbf{2} &\textbf{3}\\ \hline
\end{tabular}
\hfill
\begin{tabular}{ | c | c | c | c|c|c|}
\multicolumn{6}{c}{$C$}\\
\hline
0 & 1 & 2 & 3 & 4 & 5\\ \hline
3 & 2 & 5 & 4 & 0 &1\\ \hline
5 & 0 & 4 & 2 & 1 &3\\ \hline
4 & 5 & 1 & 0 & 3 &2\\ \hline
1 & 4 & 3 & 5 & 2 &0\\ \hline \hline
\textbf{2} & \textbf{3} &\textbf{0} & \textbf{1} &\textbf{5} &\textbf{4}\\ \hline
\end{tabular}
\end{table}

We can now check the orthogonality condition, by checking the orthogonality of 
each pair of squares. For example, the symbol pair $(1,1)$ in the 
first position of the last row of the pair $(A,B)$ already appears in the first row, 
the symbol pair $(1,2)$ in the first position of the last row of the pair $(A,C)$ 
already appears in the fourth row and the symbol pair $(1,2)$ in the first position 
of the last row of the pair $(B,C)$ already appears in the fifth row.

We conclude that the orthogonality fails on the first position in each pair of 
squares. Similarly, it can be checked that orthogonality fails in each of the 
positions in the last row.

\end{ex}

In fact this particularly strong form of maximality occurs in all seven 
instances of triples of MOLR of size $5\times 6$ (see Appendix~\ref{appendix:MOLR6}
for these seven instances). The number of $5\times 6$ 
non-isotopic triples is noteworthy, since it is known that there is no pair of 
orthogonal Latin squares of order $6$. However,  as we have seen, we can find 
three squares such that the failure of orthogonality can be confined to the 
last row of all squares.

We also note that the $7\times 7$ orthogonal triples are interesting, as 
this is the first case where there is more than one non-isotopic triple of MOLS. The 
number of non-isotopic $7 \times 7$ orthogonal triples found coincides with the
corresponding result by Egan and Wanless \cite{WanlessEgan}, where a complete 
enumeration of $t$-tuples of MOLS up to order $9$ is given. This can be taken as
an independent indication that our code is correct.

\subsection{Orthogonal Youden Rectangles}
\label{ssec_youden}

Youden \cite{You} introduced a class of particularly well balanced Latin 
rectangles, now known as either Youden rectangles or Youden ``squares''.  
A Latin rectangle is a \emph{Youden rectangle} with parameter $\lambda_{cc}$  
if every pair of columns has a constant number $\lambda_{cc}$ of symbols in common.
As is well known, divisibility considerations immediately imply that for a 
$k \times n$ Youden rectangle, the parameter $\lambda_{cc}$ must satisfy
$\lambda_{cc} = \frac{k(k-1)}{n-1}$. For fixed $k$ and $n$, there is therefore
only one possible value of $\lambda_{cc}$, which has to be an integer.

Here we will also consider a relaxed version of this concept: a Latin rectangle 
is a partially balanced Youden rectangle with parameter $\lambda_{cc}^p > 0$ if 
$\lambda_{cc}^p$ is the maximum integer such that every pair of columns intersect 
in at least $\lambda_{cc}^p$ symbols. Unlike for Youden rectangles, for partially 
balanced Youden rectangles the parameter $\lambda_{cc}^p$ is not determined uniquely 
by $k$ and $n$. The pigeonhole principle gives a lower bound 
of $2k-n$ for $\lambda_{cc}^p$ so that, for 
example, in a $3 \times 4$ array columns intersect in at least $2$ symbols, and in a
$5 \times 7$ array columns intersect in at least $3$ symbols. When the parameter 
$\lambda_{cc}$ is defined, that is, when $\frac{k(k-1)}{n-1}$ is an integer, it gives 
an upper bound for $\lambda_{cc}^p$, and in
general, $\lambda_{cc}^p \leq \frac{k(k-1)}{n-1}$. When, for some order, this upper 
and lower bound only leave one possible value for $\lambda_{cc}^p$, all Latin 
rectangles of this order will be partially balanced with this parameter.

Sets of mutually orthogonal Youden rectangles should not be confused with 
multi-layered Youden rectangles, special cases of which are known as double Youden 
rectangles and triple Youden rectangles, studied in \cite{Cl}, \cite{Pr}, \cite{cl2},
\cite{Pr2} and \cite{Pr3}, but seem to be of independent interest as designs.

As far as we know this is the first complete enumeration of mutually 
orthogonal Youden rectangles. In Table~\ref{Yt} we display the number of partially 
balanced Youden rectangles and Youden rectangles for $n$ up to 7. We note that 
removing a row from an $n \times n$ Latin square always produces an $(n-1) \times n$
Youden rectangle, so the entries in those positions in the table should come as no 
surprise, given the corresponding values in Table~\ref{Result_table_1}. We have 
included them in this table for completeness, and they also give an indication of 
the correctness of our check of the Youden property, since we actually ran the
check even for these orders.
  
For most $k$, $n$ only one parameter value is possible for partially balanced Youden 
rectangles, and thus all Latin rectangles of these orders will be partially balanced
Youden rectangles, which gives independent corroboration of the correctness of the 
check of the partially balanced Youden property.
For the partially balanced $4\times 7$ rectangles the three 
rectangles can have different parameter values and we display counts for the 
different combinations.

\begin{table}[H]
\begin{center}
\begin{tabular}{|r|r|r|r|r|}
\hline
\multirow{2}{*}{Size} &\multicolumn{2}{c|}{{Partially balanced } }& \multicolumn{2}{c|}{Youden rectangles }  \\ 
\cline{2-5}
& $\#$ & $ \lambda_{cc}^p$   & $\#$ & $\lambda_{cc}$   \\ \hline
$3 \times 4$ & 1 & 2 & 1 & 2 \\ \hline
\hline
$3 \times 5$ & 1 & 1 & 0 & - \\ \hline
$4 \times 5$ & 1 & 3 & 1 & 3 \\ \hline
\hline
$3 \times 6$ & 34 & 1 & 0 & - \\ \hline
$4 \times 6$ & 513 & 2 & 0 & - \\ \hline
$5 \times 6$ & 7 & 4 & 7 & 4 \\ \hline
\hline
$3 \times 7$ & 8 & 1 & 8 & 1 \\ \hline 
\multirow{4}{*}{$4 \times 7$} & \num{321312841} &$(1,1,1)$   & 0 & - \\ \cline{2-5}
 & \num{1795612} &$(1,1,2)$   & 0 & - \\ \cline{2-5}
 & 3\,993 &$(1,2,2)$   & 0 & - \\ \cline{2-5}
 & \num{31} &$(2,2,2)$   & \phantom{*****}31 & 2 \\ \hline
$5 \times 7$ & \num{55545} & 3 & 0 & - \\ \hline
$6 \times 7$ & 16 & 5 & 16 & 5 \\ \hline
\end{tabular}
\end{center}
\caption{The number of partially balanced Youden rectangles and Youden 
rectangles for small $n$.}
\label{Yt}
\end{table}

In Figure~\ref{YT} we display one of the mutually orthogonal triples of 
$4\times 7$ Youden rectangles.

\begin{figure}[H]
\setlength\tabcolsep{5pt}
\begin{tabular}{ | c | c | c | c|c|c|c|c | c | c | c|c|c|c|c | c | c | c|c|c|c|c|c|}
   \cline{1-7} \cline{9-15} \cline{17-23}
0 & 1 & 2 & 3 & 4 & 5 & 6 && 0 & 1 & 2 & 3 & 4 & 5 & 6 && 0 & 1 & 2 & 3 & 4 & 5 & 6\\ \cline{1-7} \cline{9-15} \cline{17-23}
6 & 5 & 4 & 2 & 3 & 1 & 0 && 5 & 6 & 3 & 4 & 1 & 0 & 2 && 4 & 2 & 1 & 6 & 0 & 3 & 5\\ \cline{1-7} \cline{9-15} \cline{17-23}
3 & 4 & 1 & 6 & 0 & 2 & 5 && 6 & 2 & 5 & 0 & 3 & 1 & 4 && 5 & 6 & 0 & 1 & 2 & 4 & 3\\ \cline{1-7} \cline{9-15} \cline{17-23}
1 & 3 & 6 & 5 & 2 & 0 & 4 && 3 & 0 & 1 & 2 & 6 & 4 & 5 && 6 & 4 & 5 & 0 & 3 & 1 & 2\\ \cline{1-7} \cline{9-15} \cline{17-23}
\end{tabular}
\caption{Three pairwise orthogonal $4\times 7$ Youden rectangles.}
\label{YT}
\end{figure}

Given that we have found a triple of pairwise orthogonal Youden rectangles it is 
natural to ask for larger tuples. Aggarwal \cite{Agg} proved that the upper bound 
$n-1$ holds for the number of mutually orthogonal Latin rectangles as well as for 
squares.

The sizes 
$(n-1) \times n$ are mostly not interesting in this regard since a set of $n-1$ 
mutually orthogonal Youden rectangles can be generated by removing one row from 
a set of $n-1$ MOLS. In particular, only the sizes $5 \times 6$,
$3 \times 7$ and $4 \times 7$ are of interest for this analysis. 

\begin{observation}
	There exists a set of four pairwise orthogonal $5\times 6$ Youden rectangles,
	a set of six pairwise orthogonal $3\times 7$ Youden rectangles, 
	and a set of six pairwise orthogonal $4\times 7$ Youden rectangles.
\end{observation}

Examples of such tuples are given in Appendix \ref{YouT}. For $n=7$ and $k=3,4$ 
this reaches the theoretical upper bound, so certainly these examples are maxima.
For $n=6$, $k=5$ it is not ruled out immediately that there could exist a 
$5$-tuple of mutually orthogonal Youden rectangles, but using a simple specialized
program, we have checked that this is not the case, and that the example we
present is actually a maximum.

\subsection{Autotopism Groups of Triples}
\label{ssec_auto}

Our next aim is to investigate the autotopism groups of the triples of MOLR.  
We present the order of the autotopism group for all triples 
of MOLR, and  track the order of the autotopism  group of those triples of 
rectangles which can be extended to triples of MOLS. 

In Table~\ref{Result_table_2} we give the maximum order of an autotopism group 
for size $k\times n$, the second largest order, and the number of non-isotopic 
triples which have groups with these cardinalities. Moreover, we also give the 
number of triples with trivial autotopism group.

\begin{table}[H]
\begin{center}
\begin{tabular}{|r|r|r|r|}
\hline
\multirow{2}{*}{Size} &\multirow{2}{*}{Trivial group} & 
\multicolumn{2}{c|}{\shortstack{Number of autotopisms/ \\Number of triples} }  \\ 
\cline{3-4}
 &   & \shortstack{Second \\ maximum} & \multicolumn{1}{c|}{Maximum}   \\ \hline
$2\times 4$  &   0   & 16/1   & 48/1     \\ \hline
$3\times 4$  &     0 & 0   & 72/1   \\ \hline
$4\times 4$  &    0  &  0  & 288/1    \\ \hline \hline
$2\times 5$  &    0  & 6/2   & 10/1     \\ \hline
$3\times 5$  &    0&  0  & 10/1    \\ \hline
$4\times 5$  &    0 &  0  & 20/1    \\ \hline
$5\times 5$  &    0 & 0   & 100/1    \\ \hline \hline
$2\times 6$  & 24  & 36/1   & 72/1     \\ \hline
$3\times 6$  & 1980  & 18/4   & 36/4    \\ \hline
$4\times 6$  & 93  & 24/1    & 36/3    \\ \hline
$5\times 6$  &  0 & 9/1   &  18/2   \\ \hline \hline
$2\times 7$   & 2\,300  & 14/3   &  42/1   \\ \hline
$3\times 7$  & \num{65822447}    & 42/1   & 63/1    \\ \hline
$4\times 7$ & \num{323002195} & 42/1 & 63/1   \\ \hline
$5\times 7$     & \num{52981} & 21/1 & 42/1     \\ \hline 
$6\times 7$ &  1     & 42/3   &126/1   \\ \hline 
$7\times 7$&   0 & 294/3  &882/1  \\ \hline 
\end{tabular}
\end{center}
\caption{The number of non-isotopic triples of MOLR with autotopism group of given orders.}
\label{Result_table_2}
\end{table}

We see that there are no triples of MOLR of order $n=4,5$ with trivial 
autotopism group. We also see that triples with maximum $k$, 
given the order $n$, namely $4\times4$, $5\times5$, $5\times6$, $7\times7$, 
always have a non-trivial autotopism group.  

In Tables \ref{Auto4},~\ref{Auto5},~\ref{Auto6}, and ~\ref{Auto7} we give 
complete statistics on the number of triples with autotopism groups of given 
orders. We see here that for $n=6,7$ having a trivial autotopism group 
becomes the most common case, though not for each $k$ separately.

Here we also track the order of the autotopism group for the triples 
of MOLR of each order which can be extended to a triple of MOLS. For $n=6$ we 
follow the triple of MOLR which can be extended to a $5\times 6$ triple of MOLR. The 
numbering of those cases correspond to the number given each triple of MOLS in 
the appendices, 
see~Appendix~\ref{appendix:MOLS4},~\ref{appendix:MOLS5},~\ref{appendix:MOLR6},~\ref{appendix:MOLS7}. Note that for orders $4$ and $5$ there is only one 
maximum example with respect to isotopism.

\begin{table}[H]
\begin{center}
\begin{tabular*}{.5\linewidth}{@{\extracolsep{\fill}}| c | c | c | }
\hline
\multicolumn{3}{|c|}{$2\times 4$}\\ \hline
\shortstack{Group cardinality}&16 & 48\\ \hline
\shortstack{Number of triples}&1 & 1 \\ \hline
MOLS case& &1  \\ \hline
\end{tabular*}

\begin{tabular*}{.5\linewidth}{@{\extracolsep{\fill}}| c |c|  }
\hline
\multicolumn{2}{|c|}{$3\times 4$}\\ \hline
\shortstack{Group cardinality}&72\\ \hline
\shortstack{Number of triples} & 1 \\ \hline
MOLS case &1  \\ \hline
\end{tabular*}

\begin{tabular*}{.5\linewidth}{@{\extracolsep{\fill}}| c |c|  }
\hline
\multicolumn{2}{|c|}{$4\times 4$}\\ \hline
\shortstack{Group cardinality}&288\\ \hline
\shortstack{Number of triples} & 1 \\ \hline
MOLS case &1  \\ \hline
\end{tabular*}
\end{center}
\caption{Autotopism groups of triples of order 4.}
\label{Auto4}
\end{table}

\begin{table}[H]
\begin{center}
\begin{tabular*}{.5\linewidth}{@{\extracolsep{\fill}}| c | c | c | c| }
\hline
\multicolumn{4}{|c|}{$2\times 5$}\\ \hline
\shortstack{Group cardinality}&2 & 6 &10\\ \hline
\shortstack{Number of triples}&1 & 2 & 1\\ \hline
MOLS case& & &1  \\ \hline
\end{tabular*}

\begin{tabular*}{.5\linewidth}{@{\extracolsep{\fill}}| c | c | }
\hline
\multicolumn{2}{|c|}{$3\times 5$}\\ \hline
\shortstack{Group cardinality} &10\\ \hline
\shortstack{Number of triples}& 1\\ \hline
MOLS case& 1  \\ \hline
\end{tabular*}

\begin{tabular*}{.5\linewidth}{@{\extracolsep{\fill}}| c | c | }
\hline
\multicolumn{2}{|c|}{$4\times 5$}\\ \hline
\shortstack{Group cardinality} &20\\ \hline
\shortstack{Number of triples}& 1\\ \hline
MOLS case& 1  \\ \hline
\end{tabular*}
\begin{tabular*}{.5\linewidth}{@{\extracolsep{\fill}}| c | c | }
\hline
\multicolumn{2}{|c|}{$5\times 5$}\\ \hline
\shortstack{Group cardinality} &100\\ \hline
\shortstack{Number of triples}& 1\\ \hline
MOLS case& 1  \\ \hline
\end{tabular*}
\end{center}
\caption{Autotopism groups of triples of order 5.}
\label{Auto5}
\end{table}

\begin{table}[H]
\begin{center}
\setlength\tabcolsep{1.5pt}
\begin{tabular*}{\linewidth}{@{\extracolsep{\fill}}| c | c | c | c | c|c|c|c|c|c|}
\hline
\multicolumn{10}{|c|}{$2\times 6$}\\ \hline
\shortstack{Group cardinality}&1 & 2 & 4 & 6 & 8 &12 & 24 & 36& 72\\ \hline
\shortstack{Number of triples}&24& 25 & 26 & 2 & 7 &13 & 4 & 1& 1\\ \hline
MOLR case& & & & & & 4,5,6,7 & 1,3 & & 2 \\ \hline
\end{tabular*}

\begin{tabular*}{\linewidth}{@{\extracolsep{\fill}} | c | c | c | c|c|c|c|c|c|}
\hline
\multicolumn{9}{|c|}{$3\times 6$}\\ \hline
\shortstack{Group cardinality}&1 & 2 &3 &4 & 6 &12 & 18 & 36\\ \hline
\shortstack{Number of triples}&1\,980& 442 & 54 & 27 & 55 &6 & 4 & 4\\ \hline
MOLR case& &  & 1&  & 2,3,4,5,6,7  & &  & \\ \hline
\end{tabular*}

\begin{tabular*}{\linewidth}{@{\extracolsep{\fill}} | c | c | c | c|c|c|c|c|c|c|c|c|}
\hline
\multicolumn{12}{|c|}{$4\times 6$}\\ \hline
\shortstack{Group cardinality}&1 & 2 &3 &4 & 6& 8 & 9 &12 &18 & 24 & 36\\ \hline
\shortstack{Number of triples}&93 & 194 &96 &37 & 64& 3 & 2 &11 &9 & 1 & 3\\ \hline
MOLR case& &  &6 & & 1,2,3,4,5 &  &  & & 7 &  & \\ \hline
\end{tabular*}

\begin{tabular*}{\linewidth}{@{\extracolsep{\fill}} | c | c | c | c|c|}
\hline
\multicolumn{5}{|c|}{$5\times 6$}\\ \hline
\shortstack{Group cardinality}&3 & 6 & 9 &18\\ \hline
\shortstack{Number of triples}&2 & 2 & 1 &2\\ \hline
MOLR case&4,6 &1,3  &7 &2,5 \\ \hline
\end{tabular*}
\end{center}
\caption{Autotopism groups of triples of order 6.}
\label{Auto6}
\end{table}

\begin{table}[H]
\begin{center}
\setlength\tabcolsep{1.5pt}
\begin{tabular*}{1.01\linewidth}{@{\extracolsep{\fill}}| c | c | c | c|c|c|c|c|c|}
\hline
\multicolumn{9}{|c|}{$2\times 7$}\\ \hline
\shortstack{Group cardinality}& 1 & 2 & 3 & 4 & 6 & 12 & 14 & 42\\ \hline
\shortstack{Number of triples}& \num{2300} & 512 & 3 & 28 & 9 & 2 & 3 & 1\\ \hline
MOLS case& & & & & &  & 1,2,4 & 3  \\ \hline
\end{tabular*}

\begin{tabular*}{1.01\linewidth}{@{\extracolsep{\fill}}|c|c|c|c|c|c|c|c|c|c|c|c|}
\hline
\multicolumn{12}{|c|}{$3\times 7$}\\ \hline
\shortstack{Group cardinality}& 1 & 2 & 3 & 6 & 7 & 9 & 14 & 18 & 21 & 42 & 63\\ \hline
\shortstack{Number of triples}& \num{65822447} & \num{60195} & 635 & 143 & 17 & 3 & 3 & 4& 4 & 1 &1\\ \hline
MOLS case& &  &  &  &  & & 1,2,4 & &  & 3 &  \\ \hline
\end{tabular*}

\begin{tabular*}{1.01\linewidth}{@{\extracolsep{\fill}}|c|c|c|c|c|c|c|c|c|c|c|c|c|c|}
\hline
\multicolumn{14}{|c|}{$4\times 7$}\\ \hline
\shortstack{Group cardinality}& 1 & 2 & 3 & 4 & 6 & 7 & 9 & 12 & 14 & 18 & 21 & 42 & 63\\ \hline
\shortstack{Number of triples}& \num{323002195} & \num{107997} & \num{1975} & 120 & 116 & 43 & 10 & 3 & 8 & 2 & 6 & 1 & 1\\ \hline
MOLS case& & & & & & & & & & & 1,2,4 & &3\\ \hline
\end{tabular*}

\begin{tabular*}{1.01\linewidth}{@{\extracolsep{\fill}}| c | c | c | c|c|c|c|c|c|c|}
\hline
\multicolumn{10}{|c|}{$5\times 7$}\\ \hline
\shortstack{Group cardinality}& 1 & 2 & 3 & 4 & 6 & 7 & 14 & 21 & 42\\ \hline
\shortstack{Number of triples}& \num{52981} & \num{2500} & 32 & 5 & 2 & 15 & 8 & 1 & 1\\ \hline
MOLS case& & & & & & & 1,2,4 & & 3\\ \hline
\end{tabular*}

\begin{tabular*}{1.01\linewidth}{@{\extracolsep{\fill}}| c | c | c | c|c|c|c|c|c|}
\hline
\multicolumn{9}{|c|}{$6\times 7$}\\ \hline
\shortstack{Group cardinality}&1 & 2 & 3 & 4 & 6 &12 & 42 & 126 \\ \hline
\shortstack{Number of triples}&1 & 4 & 1 & 1 &3& 2 &3&1\\ \hline
MOLS case& & & & & & 1,2,4 &  &3\\ \hline
\end{tabular*}

\begin{tabular*}{1.01\linewidth}{@{\extracolsep{\fill}}| c | c | c |}
\hline
\multicolumn{3}{|c|}{$7\times 7$}\\ \hline
\shortstack{Group cardinality}&294 & 882  \\ \hline
\shortstack{Number of triples}&3 & 1 \\ \hline
MOLS case&1,2,4 & 3\\ \hline
\end{tabular*}
\end{center}
\caption{Autotopism groups of triples of order 7.}
\label{Auto7}
\end{table}

We make the quite striking observation that for $n\leq 7$ the triples of MOLR which 
can be 
extended to a triple of MOLS, or MOLR with maximal $k$ for $n=6$, always have a 
non-trivial autotopism group. Given the paucity of such triples for larger $n$  
it would be interesting if this non-triviality can be proven to hold for all 
$n\geq 8$ as well, or if there are non-symmetrical examples as well, especially in 
view of the result of \cite{MMM} that in any triple of MOLS for $n=10$, each of the 
three Latin squares involved must have a trivial symmetry group.

\section{Larger orders}
\label{sec_larger}

Our methods and program can be applied to larger values of $n$ as 
well, but here the number of triples of MOLR quickly becomes unmanageable. It 
is easy to generate all non-isotopic $2\times 8$ triples, and we found that 
the number of these is:
\[
	2\times 8: \quad \num{188126}
\]
Going to $3\times 8$ becomes significantly harder, but we have done this step as 
well. Using the parallel machine this took approximately $\num{150000}$ core hours, 
and saving the full output would have required several terabytes of disk 
space. We found that the number of non-isotopic $3\times 8$ triples of MOLR 
is:
\[
	3\times 8: \quad \num{3 321 281 937 279}
\]
Given that we expect the number of $4\times 8$ triples to be significantly 
larger, we conclude that at this point it is not possible to do a full 
enumeration for $n\geq 8$ while saving the resulting triples to disk. 

However, in Subsection~\ref{ssec_auto} we noted that for $n\leq 7$ the $k\times n$ triples 
of MOLR with maximum $k$ all have non-trivial autotopism group, and that they can 
be constructed by extending triples with lower $k$ and non-trivial autotopisms. 
With this in mind we recursively define a triple of MOLR to be 
\emph{stepwise symmetric} if it has non-trivial autotopism group and is an 
extension of a stepwise symmetric triple of MOLR. We can enumerate stepwise symmetric triples 
of MOLR by starting with the unique $1\times n$ normalized triple and after each 
extension step removing the triples with trivial autotopism group. In 
Table~\ref{Result_table_18} we record the number of stepwise symmetric triples of 
MOLR for $n=8$. As we can see, the total number of non-isotopic $3\times 8$ 
triples is vastly larger than the number of stepwise symmetric such triples. As 
expected, we find that the stepwise symmetric class includes examples of triples of 
$8\times 8$ MOLS. We also note that there are such triples of squares that are not 
part of the projective plane of order 8.

\begin{table}[H]
\begin{center}
\begin{tabular}{|r|r|r|}
\hline
\multirow{2}{*}{Size}  &\multicolumn{2}{c|}{non-isotopic} \\ \cline{2-3}
&    \multicolumn{1}{c|}{All}  & \multicolumn{1}{c|}{Youden}  \\ \hline

$2\times 8$ & \num{10211}      & 0 \\ \hline
$3\times 8$ & \num{22747116}   & 0 \\ \hline
$4\times 8$ & \num{3796573635} & $\lambda_{cc}^p = 1 \quad $ \hfill \num{26747355} \\ \hline
$5\times 8$ & \num{2503469320} & $\lambda_{cc}^p = 2 \quad $ \hfill \num{2503469320} \\ \hline
$6\times 8$ & \num{5572534}    & $\lambda_{cc}^p = 4 \quad $ \hfill \num{5572534} \\ \hline
$7\times 8$ & 751 & $\lambda_{cc}=6$  \hfill 751 \\ \hline
$8\times 8$ & 72  & $\lambda_{cc}=8$  \hfill 72  \\ \hline
\end{tabular}
\end{center}
\caption{The number of stepwise symmetric triples of MOLR for $n=8$.}
\label{Result_table_18}
\end{table}

For $k=7, 8$ the  $k\times n$ rectangles are automatically Youden rectangles.
For $k=4$, the bounds on $\lambda_{cc}^p$ discussed 
in Subsection~\ref{ssec_youden} leave 
$\lambda_{cc}^p = 1$ as a possible non-zero value, and for $k=5$ and $k=6$, we 
must have $2 \leq \lambda_{cc}^p \leq \frac{20}{7}$ and 
$4 \leq \lambda_{cc}^p \leq \frac{30}{7}$ respectively, leaving only the 
possible values $\lambda_{cc}^p = 2$ and $\lambda_{cc}^p = 4$.

We have also checked the $72$ triples of MOLS for maximality with respect to adding 
further squares, and we find that $70$ of the triples are maximal. One of the two 
remaining triples extends to a maximal 4-tuple of MOLS, and the other extends to a 
full 7-tuple of MOLS, corresponding to the projective plane of this order.

As for smaller $n$, for $n=8$ we have also investigated the order of the autotopism 
groups for the stepwise symmetric triples of MOLR. The results are shown in 
Table~\ref{Result_table_28}. Here we note that for each $k\le 7$ 
the most common group order is 2, and similarly to the non-symmetric case there 
are no $8\times 8$ triples with an autotopism group of the smallest possible order, 
namely $2$. We also note that order 4 is more common than order 3, so the number
of triples is not decreasing in the order of the autotopism group.
 
\afterpage{
    \clearpage
    \begin{landscape}
        \centering 
\begin{table}[H]
\scriptsize
\begin{center}

\begin{tabular*}{1.062\linewidth}{@{\extracolsep{\fill}}| c | c | c | c|c|c|c|c|c|c|c|c|c|c|c|}
\hline
\multicolumn{15}{|c|}{$2\times 8$}\\ \hline
Autot. card.&2 & 3 & 4 & 6 & 8 &12 & 16 & 24&32& 48&64&96&128&384\\ \hline
Number of inst.& \num{9014} & 24 & 919 & 22 & 146 &14 & 46 & 2&17& 2&2&1&1&1\\ \hline
\end{tabular*}

\begin{tabular*}{1.062\linewidth}{@{\extracolsep{\fill}}| c | c | c | c|c|c|c|c|c|c|c|c|c|c|}
\hline
\multicolumn{14}{|c|}{$3\times 8$}\\ \hline
Autot. card.& 2 & 3 & 4 & 6 & 8 & 12 & 16& 18 & 24 &36&48 & 144 &576\\ \hline
Number of inst.&\num{22691810} & \num{7359} & \num{45071} & \num{1140} & \num{1476} & 117 & 117 & 4& 5 & 3 &11 &1 & 2\\ \hline
\end{tabular*}

\begin{tabular*}{1.062\linewidth}{@{\extracolsep{\fill}}|c|c|c|c|c|c|c|c|c|c|c|c|c|c|c|c|c|c|c|c|}
\hline
\multicolumn{20}{|c|}{$4\times 8$}\\ \hline
Autot. card.&2 & 3  & 4  & 6 & 8 & 9& 12 & 16 & 18 & 24& 32 & 36 &48 & 64 & 72 & 96 & 144 & 192 & 2304\\ \hline
Number of inst.&\num{3794354460} & \num{55030} & \num{2110956} & \num{8771} & \num{40767} & 38 & 542 & \num{2613} & 26 & 116 & 225 & 3 & 42 & 32 & 2 & 4 & 1 & 5 & 2\\ \hline
\end{tabular*}

\begin{tabular*}{1.062\linewidth}{@{\extracolsep{\fill}}| c | c | c | c|c|c|c|c|c|c|}
\hline
\multicolumn{10}{|c|}{$5\times 8$}\\ \hline
Autot. card.&2 & 3 & 4 & 6 & 8 &12 & 16 & 24 &32\\ \hline
Number of inst.&\num{2502639867} &418 & \num{821082} & 970 & \num{6808} &28& 134 & 7 & 6\\ \hline
\end{tabular*}

\begin{tabular*}{1.062\linewidth}{@{\extracolsep{\fill}}| c | c | c | c|c|c|c|c|c|c|c|c|}
\hline
\multicolumn{12}{|c|}{$6\times 8$}\\ \hline
Autot. card.& 2 & 4 & 6 & 8 & 12 & 16 & 18 & 24 & 36 & 48& 96 \\ \hline
Number of inst.& \num{5488623} & \num{79327} & 154 & \num{4200} & 10 & 202 & 2 & 10 & 1 & 4& 1 \\ \hline
\end{tabular*}

\begin{tabular*}{1.062\linewidth}{@{\extracolsep{\fill}}| c | c | c | c|c|c|c|c|c|c|c|c|}
\hline
\multicolumn{12}{|c|}{$7\times 8$}\\ \hline
Autot. card.& 2 & 4 & 6 & 8 & 12 & 16 & 18 & 24 & 36 & 48& 56 \\ \hline
Number of inst.& 373 & 256 & 11 & 62 & 28 & 10 & 2 & 5 & 1 & 2& 1 \\ \hline
\end{tabular*}

\begin{tabular*}{1.062\linewidth}{@{\extracolsep{\fill}}| c | c | c | c|c|c|c|c|c|c|}
\hline
\multicolumn{10}{|c|}{$8\times 8$}\\ \hline
Autot. card.&8 & 16 & 24 & 32 & 48 &64 & 96 & 192 &448\\ \hline
Number of inst.&4 &41 & 2 & 18 & 1 & 2 & 2 & 1 &1\\ \hline
\end{tabular*}

\caption{The order of the autotopism group for the stepwise symmetric triples 
	of $k\times 8$ rectangles.}
\label{Result_table_28}
\end{center}
\end{table}

    \end{landscape}
    \clearpage
}

One noteworthy feature is the number of rows in the triples with the maximum 
autotopism group order. For $n\leq 7$, with the exception of $n=6$, 
this has been achieved by triples with the maximum possible $k$, 
but when $n=8$, we instead find the maximum among the $4\times 8$ triples, 
where we find two triples with autotopism group of order $2304=2^8 \cdot 3 ^2$. These 
triples are presented in Figure~\ref{figsym}. Both of these triples are maximal 
with respect to addition of more rows.

\begin{figure}[H]
\setlength\tabcolsep{3.7pt}

\begin{subfigure}[]{1\textwidth}
\centering
\begin{tabular}{|c|c|c|c||c|c|c|c| c |c|c|c|c||c|c|c|c| c |c|c|c|c||c|c|c|c|}
  \cline{1-8} \cline{10-17} \cline{19-26} 
0 &  1 &  2 &  3 &  4 &  5 &  6 &  7 &&  0 &  1 &  2 &  3 &  4 &  5 &  6 &  7 &&  0 &  1 &  2 &  3 &  4 &  5 &  6 &  7 \\ \cline{1-8} \cline{10-17} \cline{19-26} 

3 &  2 &  1 &  0 &  7 &  6 &  5 &  4 &&  2 &  3 &  0 &  1 &  6 &  7 &  4 &  5 &&  1 &  0 &  3 &  2 &  5 &  4 &  7 &  6 \\ \cline{1-8} \cline{10-17} \cline{19-26} 

2 &  3 &  0 &  1 &  6 &  7 &  4 &  5 &&  1 &  0 &  3 &  2 &  5 &  4 &  7 &  6 &&  3 &  2 &  1 &  0 &  7 &  6 &  5 &  4 \\ \cline{1-8} \cline{10-17} \cline{19-26} 

1 &  0 &  3 &  2 &  5 &  4 &  7 &  6 &&  3 &  2 &  1 &  0 &  7 &  6 &  5 &  4 &&  2 &  3 &  0 &  1 &  6 &  7 &  4 &  5 \\ 
\cline{1-8} \cline{10-17} \cline{19-26} 
\end{tabular}
\subcaption{First triple of mutually orthogonal $8 \times 4$ MOLR.}
\label{triple1}
\end{subfigure}
\vspace*{2em}

\begin{subfigure}[]{1 \textwidth}
\centering
\begin{tabular}{|c|c|c|c||c|c|c|c| c |c|c|c|c||c|c|c|c| c |c|c|c|c||c|c|c|c|}
  \cline{1-8} \cline{10-17} \cline{19-26} 
0 &  1 &  2 &  3 &  4 &  5 &  6 &  7  &&  0 &  1 &  2 &  3 &  4 &  5 &  6 &  7 &&  0 &  1 &  2 &  3 &  4 &  5 &  6 &  7 \\ \cline{1-8} \cline{10-17} \cline{19-26} 

3 &  2 &  1 &  0 &  7 &  6 &  5 &  4  &&  2 &  3 &  0 &  1 &  6 &  7 &  4 &  5 &&  1 &  0 &  3 &  2 &  5 &  4 &  7 &  6 \\ \cline{1-8} \cline{10-17} \cline{19-26} 

2 &  3 &  0 &  1 &  5 &  4 &  7 &  6  &&  1 &  0 &  3 &  2 &  7 &  6 &  5 &  4 &&  3 &  2 &  1 &  0 &  6 &  7 &  4 &  5 \\ \cline{1-8} \cline{10-17} \cline{19-26} 

1 &  0 &  3 &  2 &  6 &  7 &  4 &  5  &&  3 &  2 &  1 &  0 &  5 &  4 &  7 &  6 &&  2 &  3 &  0 &  1 &  7 &  6 &  5 &  4 \\ \cline{1-8} \cline{10-17} \cline{19-26}
\end{tabular}
\subcaption{Second triple of mutually orthogonal $8 \times 4$ MOLR.}
\label{triple2}
\end{subfigure}

\caption{The two $4\times 8$ triples with autotopism group of order 2304.}
\label{figsym}
\end{figure}

As pointed out to us by Rosemary Bailey these examples can be generated 
from cosets of the elementary Abelian group of order 32 (see \cite{RB}), or even simpler, 
by juxtaposing two triples of MOLS of order 4. The double line in the above
examples indicates this juxtaposition, and we have chosen representatives from the 
isotopism classes of triples of $4 \times 4$ MOLS so that the left side of both the 
above examples of triples of $4 \times 8$ MOLR corresponds 
exactly to our chosen representative of the isotopic triples of MOLS of order 4. 
In Figure~\ref{triple1}, the right side also corresponds exactly to the same triple of 
$4 \times 4$ MOLS, by renaming the symbols $4,5,6,7$ to $0,1,2,3$, respectively. In 
Figure~\ref{triple2}, rows 3 and 4 in the $4 \times 4$ MOLS on the right side are switched.

Based on this observation, we formulate the following result, which follows
immediately from juxtaposing the two $t$-tuples of MOLR, and introducing new
symbol names for the triple of MOLR on the right side.

\begin{proposition}\label{prop:juxta}
	If there exists a $t$-tuple of $k \times n_1$ MOLR, and a $t$-tuple of 
	$k \times n_2$ MOLR, then there exists a $t$-tuple of $k \times (n_1 + n_2)$ 
	MOLR.
\end{proposition}

For example, there exists a $4$-tuple of $5 \times 10$ MOLR, since there 
exists a $4$-tuple of $5 \times 5$ MOLR. However, the triples of 
MOLR formed by applying 
Proposition~\ref{prop:juxta} will clearly have poor properties regarding adding 
more rows.

\section*{Acknowledgments}

The computational work was performed on resources provided by the 
Swedish National Infrastructure for Computing (SNIC) at 
High Performance Computing Center North (HPC2N).
This work was supported by the Swedish strategic research programme eSSENCE.    
This work was supported by The Swedish Research Council grant 2014-4897.

\newcommand{\etalchar}[1]{$^{#1}$}
\providecommand{\bysame}{\leavevmode\hbox to3em{\hrulefill}\thinspace}
\providecommand{\MR}{\relax\ifhmode\unskip\space\fi MR }
\providecommand{\MRhref}[2]{%
  \href{http://www.ams.org/mathscinet-getitem?mr=#1}{#2}
}
\providecommand{\href}[2]{#2}

\newpage

\appendix

\section{The unique (up to isotopism) triple of \boldmath{$4\times 4$} MOLS}
\label{appendix:MOLS4}

\begin{table}[H]
\centering{
\begin{tabular}{  | c | c|c|c|c|c |  c|c|c|c|c | c|c|c|}
\cline{1-4} \cline{6-9} \cline{11-14} 
0& 1& 2& 3& &0& 1& 2& 3& &0& 1& 2& 3\\ \cline{1-4} \cline{6-9} \cline{11-14}
3& 2& 1& 0& &2& 3& 0& 1& &1& 0& 3& 2\\ \cline{1-4} \cline{6-9} \cline{11-14}
2& 3& 0& 1& &1& 0& 3& 2& &3& 2& 1& 0\\ \cline{1-4} \cline{6-9} \cline{11-14}
1& 0& 3& 2& &3& 2& 1& 0& &2& 3& 0& 1\\ \cline{1-4} \cline{6-9} \cline{11-14}
\end{tabular}}
\label{MOLS-4}
\end{table}

\section{The unique (up to isotopism) triple of \boldmath{$5\times 5$} MOLS}
\label{appendix:MOLS5}

\begin{table}[H]
\centering{
\begin{tabular}{ | c | c | c|c|c|c|c | c|  c|c|c|c|c | c| c|c|c|}
\cline{1-5} \cline{7-11} \cline{13-17}
0&1&2&3&4&  &0&1&2&3&4&  &0&1&2&3&4\\ \cline{1-5} \cline{7-11} \cline{13-17}
4&3&1&0&2&  &3&2&4&1&0&  &2&0&3&4&1\\ \cline{1-5} \cline{7-11} \cline{13-17}
3&2&4&1&0&  &4&3&1&0&2&  &1&4&0&2&3\\ \cline{1-5} \cline{7-11} \cline{13-17}
2&0&3&4&1&  &1&4&0&2&3&  &3&2&4&1&0\\ \cline{1-5} \cline{7-11} \cline{13-17}
1&4&0&2&3&  &2&0&3&4&1&  &4&3&1&0&2\\ \cline{1-5} \cline{7-11} \cline{13-17}
\end{tabular}}
\label{MOLS-5}
\end{table}

\section{The seven non-isotopic triples of \boldmath{$5\times 6$} MOLR}
\label{appendix:MOLR6}

\begin{enumerate}
\item \hspace*{0em}

\begin{table}[H]
\begin{tabular}{ | c | c | c | c|c|c|c|c | c | c | c|c|c|c|c | c | c | c|c|c|}
\cline{1-6} \cline{8-13} \cline{15-20} 
0& 1& 2& 3& 4& 5&& 0& 1& 2& 3& 4& 5&& 0& 1& 2& 3& 4& 5\\ \cline{1-6} \cline{8-13} \cline{15-20} 
5& 4& 3& 2& 1& 0&& 4& 5& 1& 0& 3& 2&& 3& 2& 5& 4& 0& 1\\ \cline{1-6} \cline{8-13} \cline{15-20} 
4& 5& 1& 0& 3& 2&& 3& 2& 5& 4& 0& 1&& 5& 0& 4& 2& 1& 3\\ \cline{1-6} \cline{8-13} \cline{15-20} 
3& 2& 5& 4& 0& 1&& 2& 4& 3& 1& 5& 0&& 4& 5& 1& 0& 3& 2\\ \cline{1-6} \cline{8-13} \cline{15-20} 
2& 0& 4& 1& 5& 3&& 5& 3& 0& 2& 1& 4&& 1& 4& 3& 5& 2& 0\\ \cline{1-6} \cline{8-13} \cline{15-20} 
\end{tabular}
\end{table}

\item \hspace*{0em}
\begin{table}[H]
\begin{tabular}{ | c | c | c | c|c|c|c|c | c | c | c|c|c|c|c | c | c | c|c|c|}
\cline{1-6} \cline{8-13} \cline{15-20} 
0& 1& 2& 3& 4& 5&& 0& 1& 2& 3& 4& 5&& 0& 1& 2& 3& 4& 5\\ \cline{1-6} \cline{8-13} \cline{15-20} 
5& 4& 3& 2& 1& 0&& 4& 3& 5& 1& 0& 2&& 3& 5& 4& 0& 2& 1\\ \cline{1-6} \cline{8-13} \cline{15-20} 
4& 3& 5& 1& 0& 2&& 5& 4& 3& 2& 1& 0&& 2& 0& 1& 4& 5& 3\\ \cline{1-6} \cline{8-13} \cline{15-20} 
3& 5& 4& 0& 2& 1&& 2& 0& 1& 4& 5& 3&& 5& 4& 3& 2& 1& 0\\ \cline{1-6} \cline{8-13} \cline{15-20} 
2& 0& 1& 4& 5& 3&& 3& 5& 4& 0& 2& 1&& 4& 3& 5& 1& 0& 2\\ \cline{1-6} \cline{8-13} \cline{15-20}  \cline{15-20} 
\end{tabular}
\end{table}
\pagebreak[4]

\item \hspace*{0em}
\begin{table}[H]
\begin{tabular}{ | c | c | c | c|c|c|c|c | c | c | c|c|c|c|c | c | c | c|c|c|}
\cline{1-6} \cline{8-13} \cline{15-20} 
0& 1& 2& 3& 4& 5&& 0& 1& 2& 3& 4& 5&& 0& 1& 2& 3& 4& 5\\ \cline{1-6} \cline{8-13} \cline{15-20} 
5& 4& 3& 2& 1& 0&& 4& 3& 5& 0& 2& 1&& 3& 5& 4& 1& 0& 2\\ \cline{1-6} \cline{8-13} \cline{15-20} 
4& 3& 5& 0& 2& 1&& 5& 4& 3& 2& 1& 0&& 2& 0& 1& 4& 5& 3\\ \cline{1-6} \cline{8-13} \cline{15-20} 
3& 5& 4& 1& 0& 2&& 2& 0& 1& 4& 5& 3&& 5& 4& 3& 2& 1& 0\\ \cline{1-6} \cline{8-13} \cline{15-20} 
1& 2& 0& 5& 3& 4&& 3& 5& 4& 1& 0& 2&& 4& 3& 5& 0& 2& 1\\ \cline{1-6} \cline{8-13} \cline{15-20}
\end{tabular}
\end{table}

\item \hspace*{0em}
\begin{table}[H]
\begin{tabular}{ | c | c | c | c|c|c|c|c | c | c | c|c|c|c|c | c | c | c|c|c|}
\cline{1-6} \cline{8-13} \cline{15-20} 
0& 1& 2& 3& 4& 5&& 0& 1& 2& 3& 4& 5&& 0& 1& 2& 3& 4& 5\\ \cline{1-6} \cline{8-13} \cline{15-20} 
5& 4& 3& 1& 0& 2&& 4& 3& 5& 2& 1& 0&& 2& 0& 1& 5& 3& 4\\ \cline{1-6} \cline{8-13} \cline{15-20} 
4& 3& 5& 2& 1& 0&& 1& 2& 0& 5& 3& 4&& 5& 4& 3& 0& 2& 1\\ \cline{1-6} \cline{8-13} \cline{15-20} 
2& 0& 1& 5& 3& 4&& 3& 5& 4& 1& 0& 2&& 1& 2& 0& 4& 5& 3\\ \cline{1-6} \cline{8-13} \cline{15-20} 
1& 2& 0& 4& 5& 3&& 5& 4& 3& 0& 2& 1&& 4& 3& 5& 2& 1& 0\\ \cline{1-6} \cline{8-13} \cline{15-20} 
\end{tabular}
\end{table}

\item \hspace*{0em}
\begin{table}[H]
\begin{tabular}{ | c | c | c | c|c|c|c|c | c | c | c|c|c|c|c | c | c | c|c|c|}
\cline{1-6} \cline{8-13} \cline{15-20} 
0& 1& 2& 3& 4& 5&& 0& 1& 2& 3& 4& 5&& 0& 1& 2& 3& 4& 5\\ \cline{1-6} \cline{8-13} \cline{15-20} 
5& 4& 3& 1& 0& 2&& 4& 3& 5& 2& 1& 0&& 2& 0& 1& 5& 3& 4\\ \cline{1-6} \cline{8-13} \cline{15-20} 
4& 3& 5& 2& 1& 0&& 1& 2& 0& 5& 3& 4&& 5& 4& 3& 0& 2& 1\\ \cline{1-6} \cline{8-13} \cline{15-20} 
2& 0& 1& 5& 3& 4&& 3& 5& 4& 1& 0& 2&& 1& 2& 0& 4& 5& 3\\ \cline{1-6} \cline{8-13} \cline{15-20} 
1& 2& 0& 4& 5& 3&& 5& 4& 3& 0& 2& 1&& 3& 5& 4& 2& 1& 0\\ \cline{1-6} \cline{8-13} \cline{15-20}
\end{tabular}
\end{table}

\item \hspace*{0em}
\begin{table}[H]
\begin{tabular}{ | c | c | c | c|c|c|c|c | c | c | c|c|c|c|c | c | c | c|c|c|}
\cline{1-6} \cline{8-13} \cline{15-20} 
0& 1& 2& 3& 4& 5&& 0& 1& 2& 3& 4& 5&& 0& 1& 2& 3& 4& 5\\ \cline{1-6} \cline{8-13} \cline{15-20} 
5& 4& 3& 1& 0& 2&& 4& 3& 5& 2& 1& 0&& 2& 0& 1& 5& 3& 4\\ \cline{1-6} \cline{8-13} \cline{15-20} 
4& 3& 5& 0& 2& 1&& 2& 0& 1& 5& 3& 4&& 1& 2& 0& 4& 5& 3\\ \cline{1-6} \cline{8-13} \cline{15-20} 
3& 5& 4& 2& 1& 0&& 1& 2& 0& 4& 5& 3&& 5& 4& 3& 1& 0& 2\\ \cline{1-6} \cline{8-13} \cline{15-20} 
1& 2& 0& 4& 5& 3&& 3& 5& 4& 1& 0& 2&& 4& 3& 5& 2& 1& 0\\ \cline{1-6} \cline{8-13} \cline{15-20} 
\end{tabular}
\end{table}

\item \hspace*{0em}
\begin{table}[H]
\begin{tabular}{ | c | c | c | c|c|c|c|c | c | c | c|c|c|c|c | c | c | c|c|c|}
\cline{1-6} \cline{8-13} \cline{15-20} 
0& 1& 2& 3& 4& 5&& 0& 1& 2& 3& 4& 5&& 0& 1& 2& 3& 4& 5\\ \cline{1-6} \cline{8-13} \cline{15-20} 
5& 4& 3& 1& 0& 2&& 4& 3& 5& 2& 1& 0&& 2& 0& 1& 5& 3& 4\\ \cline{1-6} \cline{8-13} \cline{15-20} 
4& 3& 5& 0& 2& 1&& 2& 0& 1& 5& 3& 4&& 1& 2& 0& 4& 5& 3\\ \cline{1-6} \cline{8-13} \cline{15-20} 
3& 5& 4& 2& 1& 0&& 1& 2& 0& 4& 5& 3&& 5& 4& 3& 0& 2& 1\\ \cline{1-6} \cline{8-13} \cline{15-20} 
2& 0& 1& 5& 3& 4&& 5& 4& 3& 0& 2& 1&& 3& 5& 4& 1& 0& 2\\ \cline{1-6} \cline{8-13} \cline{15-20}
\end{tabular}
\end{table}
\end{enumerate}

\newpage

\section{The four non-isotopic triples of \boldmath{$7\times 7$} MOLS}
\label{appendix:MOLS7}

\begin{enumerate}
\item \hspace*{0em}
\begin{table}[H]
\begin{tabular}{ | c | c | c | c|c|c|c|c | c | c | c|c|c|c|c | c | c | c|c|c|c|c|c|}
\cline{1-7} \cline{9-15} \cline{17-23}
0&1&2&3&4&5&6& &0&1&2&3&4&5&6& &0&1&2&3&4&5&6\\ \cline{1-7} \cline{9-15} \cline{17-23}
6&5&4&2&1&0&3& &5&4&3&6&2&1&0& &4&3&0&5&6&2&1\\ \cline{1-7} \cline{9-15} \cline{17-23}
5&4&3&6&2&1&0& &6&5&4&2&1&0&3& &2&6&5&1&0&3&4\\ \cline{1-7} \cline{9-15} \cline{17-23}
4&3&0&5&6&2&1& &2&6&5&1&0&3&4& &3&0&1&4&5&6&2\\ \cline{1-7} \cline{9-15} \cline{17-23}
3&0&1&4&5&6&2& &1&2&6&0&3&4&5& &6&5&4&2&1&0&3\\ \cline{1-7} \cline{9-15} \cline{17-23}
2&6&5&1&0&3&4& &4&3&0&5&6&2&1& &1&2&6&0&3&4&5\\ \cline{1-7} \cline{9-15} \cline{17-23}
1&2&6&0&3&4&5& &3&0&1&4&5&6&2& &5&4&3&6&2&1&0\\ \cline{1-7} \cline{9-15} \cline{17-23}
\end{tabular}
\end{table}

\item \hspace*{0em}
\begin{table}[H]
\begin{tabular}{ | c | c | c | c|c|c|c|c | c | c | c|c|c|c|c | c | c | c|c|c|c|c|c|}
\cline{1-7} \cline{9-15} \cline{17-23}
0&1&2&3&4&5&6& &0&1&2&3&4&5&6& &0&1&2&3&4&5&6\\ \cline{1-7} \cline{9-15} \cline{17-23}
6&5&4&2&1&0&3& &5&4&3&6&2&1&0& &3&0&1&4&5&6&2\\ \cline{1-7} \cline{9-15} \cline{17-23}
5&4&3&6&2&1&0& &6&5&4&2&1&0&3& &1&2&6&0&3&4&5\\ \cline{1-7} \cline{9-15} \cline{17-23}
4&3&0&5&6&2&1& &2&6&5&1&0&3&4& &6&5&4&2&1&0&3\\ \cline{1-7} \cline{9-15} \cline{17-23}
3&0&1&4&5&6&2& &1&2&6&0&3&4&5& &4&3&0&5&6&2&1\\ \cline{1-7} \cline{9-15} \cline{17-23}
2&6&5&1&0&3&4& &4&3&0&5&6&2&1& &5&4&3&6&2&1&0\\ \cline{1-7} \cline{9-15} \cline{17-23}
1&2&6&0&3&4&5& &3&0&1&4&5&6&2& &2&6&5&1&0&3&4\\ \cline{1-7} \cline{9-15} \cline{17-23}
\end{tabular}
\end{table}

\item \hspace*{0em}
\begin{table}[H]
\begin{tabular}{ | c | c | c | c|c|c|c|c | c | c | c|c|c|c|c | c | c | c|c|c|c|c|c|}
\cline{1-7} \cline{9-15} \cline{17-23}
0&1&2&3&4&5&6& &0&1&2&3&4&5&6& &0&1&2&3&4&5&6\\ \cline{1-7} \cline{9-15} \cline{17-23}
6&5&4&2&1&0&3& &4&3&0&5&6&2&1& &3&0&1&4&5&6&2\\ \cline{1-7} \cline{9-15} \cline{17-23}
5&4&3&6&2&1&0& &2&6&5&1&0&3&4& &1&2&6&0&3&4&5\\ \cline{1-7} \cline{9-15} \cline{17-23}
4&3&0&5&6&2&1& &3&0&1&4&5&6&2& &6&5&4&2&1&0&3\\ \cline{1-7} \cline{9-15} \cline{17-23}
3&0&1&4&5&6&2& &6&5&4&2&1&0&3& &4&3&0&5&6&2&1\\ \cline{1-7} \cline{9-15} \cline{17-23}
2&6&5&1&0&3&4& &1&2&6&0&3&4&5& &5&4&3&6&2&1&0\\ \cline{1-7} \cline{9-15} \cline{17-23}
1&2&6&0&3&4&5& &5&4&3&6&2&1&0& &2&6&5&1&0&3&4\\ \cline{1-7} \cline{9-15} \cline{17-23}
\end{tabular}
\end{table}

\item \hspace*{0em}
\begin{table}[H]
\begin{tabular}{ | c | c | c | c|c|c|c|c | c | c | c|c|c|c|c | c | c | c|c|c|c|c|c|}
\cline{1-7} \cline{9-15} \cline{17-23}
0&1&2&3&4&5&6& &0&1&2&3&4&5&6& &0&1&2&3&4&5&6\\ \cline{1-7} \cline{9-15} \cline{17-23}
6&5&4&2&1&0&3& &4&3&0&5&6&2&1& &1&2&6&0&3&4&5\\ \cline{1-7} \cline{9-15} \cline{17-23}
5&4&3&6&2&1&0& &2&6&5&1&0&3&4& &3&0&1&4&5&6&2\\ \cline{1-7} \cline{9-15} \cline{17-23}
4&3&0&5&6&2&1& &3&0&1&4&5&6&2& &5&4&3&6&2&1&0\\ \cline{1-7} \cline{9-15} \cline{17-23}
3&0&1&4&5&6&2& &6&5&4&2&1&0&3& &2&6&5&1&0&3&4\\ \cline{1-7} \cline{9-15} \cline{17-23}
2&6&5&1&0&3&4& &1&2&6&0&3&4&5& &6&5&4&2&1&0&3\\ \cline{1-7} \cline{9-15} \cline{17-23}
1&2&6&0&3&4&5& &5&4&3&6&2&1&0& &4&3&0&5&6&2&1\\ \cline{1-7} \cline{9-15} \cline{17-23}
\end{tabular}
\end{table}

 \end{enumerate}

\section{Maximum sets of pairwise orthogonal Youden rectangles}
\label{YouT}

\bigskip
\subsection{An example of a 4-tuple of mutually orthogonal \boldmath{$5 \times 6$} Youden rectangles}
\begin{table}[H]
\setlength\tabcolsep{5pt}
\begin{tabular}{|c|c|c|c|c|c|c |c|c|c|c|c|c|c |c|c|c|c|c|c|c |c|c|c|c|c|c|c|}
\cline{1-6} \cline{8-13} \cline{15-20} \cline{22-27}
0 & 1 & 2 & 3 & 4 & 5 && 0 & 1 & 2 & 3 & 4 & 5 && 0 & 1 & 2 & 3 & 4 & 5 && 0 & 1 & 2 & 3 & 4 & 5\\ \cline{1-6} \cline{8-13} \cline{15-20} \cline{22-27}
5 & 4 & 3 & 2 & 1 & 0 && 4 & 5 & 1 & 0 & 3 & 2 && 3 & 2 & 5 & 4 & 0 & 1 && 2 & 0 & 4 & 1 & 5 & 3\\ \cline{1-6} \cline{8-13} \cline{15-20} \cline{22-27}
4 & 5 & 1 & 0 & 3 & 2 && 3 & 2 & 5 & 4 & 0 & 1 && 5 & 0 & 4 & 2 & 1 & 3 && 1 & 4 & 3 & 5 & 2 & 0\\ \cline{1-6} \cline{8-13} \cline{15-20} \cline{22-27}
3 & 2 & 5 & 4 & 0 & 1 && 2 & 4 & 3 & 1 & 5 & 0 && 4 & 5 & 1 & 0 & 3 & 2 && 5 & 3 & 0 & 2 & 1 & 4\\ \cline{1-6} \cline{8-13} \cline{15-20} \cline{22-27}
2 & 0 & 4 & 1 & 5 & 3 && 5 & 3 & 0 & 2 & 1 & 4 && 1 & 4 & 3 & 5 & 2 & 0 && 4 & 2 & 5 & 0 & 3 & 1\\ \cline{1-6} \cline{8-13} \cline{15-20} \cline{22-27}
\end{tabular}
\end{table}

\noindent
\subsection{An example of a 6-tuple of mutually orthogonal \boldmath{$3 \times 7$} Youden rectangles}
\begin{table}[H]
\setlength\tabcolsep{5pt}
\begin{tabular}{|c|c|c|c|c|c|c| c |c|c|c|c|c|c|c| c |c|c|c|c|c|c|c|}
\cline{1-7} \cline{9-15} \cline{17-23}
0 & 1 & 2 & 3 & 4 & 5 & 6 && 0 & 1 & 2 & 3 & 4 & 5 & 6 && 0 & 1 & 2 & 3 & 4 & 5 & 6\\ \cline{1-7} \cline{9-15} \cline{17-23}
6 & 5 & 4 & 2 & 1 & 0 & 3 && 5 & 4 & 3 & 6 & 2 & 1 & 0 && 4 & 3 & 0 & 5 & 6 & 2 & 1\\ \cline{1-7} \cline{9-15} \cline{17-23}
2 & 6 & 5 & 1 & 0 & 3 & 4 && 3 & 0 & 1 & 4 & 5 & 6 & 2 && 1 & 2 & 6 & 0 & 3 & 4 & 5\\ \cline{1-7} \cline{9-15} \cline{17-23}
\end{tabular}
\\ \\ \\
\setlength\tabcolsep{5pt}
\begin{tabular}{|c|c|c|c|c|c|c| c |c|c|c|c|c|c|c| c |c|c|c|c|c|c|c|}
\cline{1-7} \cline{9-15} \cline{17-23}
0 & 1 & 2 & 3 & 4 & 5 & 6 && 0 & 1 & 2 & 3 & 4 & 5 & 6 && 0 & 1 & 2 & 3 & 4 & 5 & 6\\ \cline{1-7} \cline{9-15} \cline{17-23}
3 & 0 & 1 & 4 & 5 & 6 & 2 && 2 & 6 & 5 & 1 & 0 & 3 & 4 && 1 & 2 & 6 & 0 & 3 & 4 & 5\\ \cline{1-7} \cline{9-15} \cline{17-23}
5 & 4 & 3 & 6 & 2 & 1 & 0 && 6 & 5 & 4 & 2 & 1 & 0 & 3 && 4 & 3 & 0 & 5 & 6 & 2 & 1\\ \cline{1-7} \cline{9-15} \cline{17-23}
\end{tabular}
\end{table}

\noindent
\subsection{An example of a 6-tuple of mutually orthogonal \boldmath{$4 \times 7$} Youden rectangles}
\begin{table}[H]
\setlength\tabcolsep{5pt}
\begin{tabular}{|c|c|c|c|c|c|c| c |c|c|c|c|c|c|c| c |c|c|c|c|c|c|c|}
\cline{1-7} \cline{9-15} \cline{17-23}
0 & 1 & 2 & 3 & 4 & 5 & 6 && 0 & 1 & 2 & 3 & 4 & 5 & 6 && 0 & 1 & 2 & 3 & 4 & 5 & 6\\ \cline{1-7} \cline{9-15} \cline{17-23}
6 & 5 & 4 & 2 & 1 & 0 & 3 && 5 & 4 & 3 & 6 & 2 & 1 & 0 && 4 & 3 & 0 & 5 & 6 & 2 & 1\\ \cline{1-7} \cline{9-15} \cline{17-23}
5 & 4 & 3 & 6 & 2 & 1 & 0 && 6 & 5 & 4 & 2 & 1 & 0 & 3 && 2 & 6 & 5 & 1 & 0 & 3 & 4\\ \cline{1-7} \cline{9-15} \cline{17-23}
4 & 3 & 0 & 5 & 6 & 2 & 1 && 2 & 6 & 5 & 1 & 0 & 3 & 4 && 3 & 0 & 1 & 4 & 5 & 6 & 2\\ \cline{1-7} \cline{9-15} \cline{17-23}
\end{tabular}
\\ \\ \\
\setlength\tabcolsep{5pt}
\begin{tabular}{|c|c|c|c|c|c|c| c |c|c|c|c|c|c|c| c |c|c|c|c|c|c|c|}
\cline{1-7} \cline{9-15} \cline{17-23}
0 & 1 & 2 & 3 & 4 & 5 & 6 && 0 & 1 & 2 & 3 & 4 & 5 & 6 && 0 & 1 & 2 & 3 & 4 & 5 & 6\\ \cline{1-7} \cline{9-15} \cline{17-23}
3 & 0 & 1 & 4 & 5 & 6 & 2 && 2 & 6 & 5 & 1 & 0 & 3 & 4 && 1 & 2 & 6 & 0 & 3 & 4 & 5\\ \cline{1-7} \cline{9-15} \cline{17-23}
1 & 2 & 6 & 0 & 3 & 4 & 5 && 4 & 3 & 0 & 5 & 6 & 2 & 1 && 3 & 0 & 1 & 4 & 5 & 6 & 2\\ \cline{1-7} \cline{9-15} \cline{17-23}
6 & 5 & 4 & 2 & 1 & 0 & 3 && 1 & 2 & 6 & 0 & 3 & 4 & 5 && 5 & 4 & 3 & 6 & 2 & 1 & 0\\ \cline{1-7} \cline{9-15} \cline{17-23}
\end{tabular}
\end{table}

\newpage

\section{Algorithms}
Here we give a more detailed description of the algorithms used.

\subsection{Extension Finding}
\label{algo1}

The algorithm to find all normalized non-isotopic triples of MOLR 
of size $k\times n$ is recursive and based on a breadth first search. It
takes all normalized non-isotopic triples of MOLR of size $(k-1)\times 
n$ as input, finds an extension of size $k\times n$
and saves the result in a file. Note that nothing in the algorithm requires
that the input rectangles are normalized, but in practice, we only used normalized
rectangles as input.

It should also be noted that the algorithm will not necessarily produce normalized 
triples, and it may also produce many isotopic triples. The second part, described 
in the next subsection, picks out non-isotopic triples and outputs 
each such representative in normalized form.

The search when adding the $k$-th rows of the rectangles starts with 
the lexicographically largest permutation and goes to the smallest one, 
i.e., the search is done in reverse order.

Clearly, only a small fraction of all $n!$ possible permutations can be used as 
the $k$-th row of the Latin rectangles. Thus, to speed up the 
algorithm in practice, for each triple of MOLR we generate three lists of permutations, 
$A_k$, $B_k$, $C_k$ that are possible as $k$-th rows of the corresponding Latin 
rectangles and use only those. For the purposes of keeping the pseudocode 
simple, this restriction is not reflected, and all permutations are taken
from $S_n$, the full set of permutations on $n$ elements.   

Pseudocode for the process of finding an extension of a triple of MOLR is 
given in Algorithm~\ref{algo1_ext}.

\bigskip

\begin{algorithm}[H]
\label{algo1_ext}
\caption{Extension of a triple of MOLR.}
  \KwIn{$(k-1)\times n$ normalized triple of MOLR $(A,B,C)$. }
  \KwOut{ List $L$ of $k\times n$ triples of MOLR.}
\Begin{
    $L=\emptyset$.\\
 	\ForEach {  $\alpha \in S_n$} 
 		{Extend  rectangle $A$ by $\alpha$ and check the Latin rectangle condition.\\
 		\ForEach{$\beta \in S_n$} 
 			{Extend  rectangle $B$ by $\beta$ and check the Latin rectangle condition.\\
 			Check the orthogonality of $(A,B)$.\\
 			\ForEach{ $\gamma \in S_n$} 
 				{Extend  rectangle $C$ by $\gamma$ and check the Latin rectangle condition.\\
 				Check the orthogonality  of  $(A,C)$ and of $(B,C)$.\\
				If all checks are positive, add extended triple to $L$.
 			}
 		}
	}  	
}

\end{algorithm}
{\ }\\

\subsection{Isotopism Rejection}
\label{algo2}

Algorithm~\ref{algo1_ext} may create many isotopic triples of MOLR and the aim of our 
second step is to keep exactly one representative for each class of 
isotopic triples of MOLR, and to output that representative in normalized form. As 
the representative of an isotopism class we will choose the triple which is highest 
in the lexicographic order. The main idea for the isotopism rejection step is thus to 
check whether it is possible to transform a triple $(A,B,C)$ of MOLR to an isotopic 
triple with higher lexicographic order, and if that is possible, the triple $(A,B,C)$ 
is discarded. This test is applied to each generated triple in order to determine 
whether it should be output or not.  

The isotopism rejection is based on all six types of permutations of a 
triple of MOLR of size $k\times n$, namely permuting rectangles, permuting rows, 
permuting columns, and permuting symbols separately in each rectangle. 
Therefore, in general, it requires six nested loops and $3!\times k! \times n! 
\times n!^3$ steps.



We use the concept of normalized triples of MOLR to reduce the complexity of 
this algorithm. First, the first rows
are transformed to identity permutations by permuting symbols. 
As this transformation is unique once the column permutation 
$\sigma$ is chosen, we skip the loops with symbol permutations.

Second, since the Latin rectangles in a normalized triple of MOLR are ordered 
(condition (S2)), it is sufficient to choose which of the rectangles is the 
first one (three options) and reorder the last $k-1$ rows to satisfy condition 
(S3). The other two rectangles will be ordered by the second row.

Finally, rows are ordered by the first rectangle, therefore,  we just need to 
choose which of the rows is the first one ($k$ options).  There is no 
restriction for permuting columns, so we keep this loop as it is. In summary, 
by using normalization, we have reduced the number of steps in the algorithm 
to $3\times k\times n!$.

Pseudocode for this step is given in Algorithm~\ref{algo2_iso}.
In the pseudocode, $M_2$ denotes the second row of a matrix $M$, and the 
lexicographically largest permutation in a set $\alpha$ of permutations is 
denoted by $\max\{\alpha\}$.

\bigskip

\begin{algorithm}[H]
\label{algo2_iso}
\caption{Isotopism rejection of a list of triples of MOLR.}
  \KwIn{List $L$ of $k\times n$ triples of MOLR. }
  \KwOut{List $L^\prime$ of normalized non-isotopic $k\times n$ triples of 
	MOLR.}
\Begin{
	$L^\prime=\emptyset$.\\
	\ForEach {$(A,B,C) \in  L$ }
 	{\ForEach {$D \in \{A,B,C\}$} 
 		{\ForEach{$i \in\{1,2,\dots ,k\}$} 
 			{\ForEach{$\sigma \in S_n$} 
 				{
 				 $(X,Y,Z):=(A,B,C)$.\\
 				 Apply column permutation $\sigma$ for the rectangles $X,Y,Z$.\\
 				 Apply symbol permutation such that the $i$-th rows are the identity permutations.\\
 				 Order rows such that the $i$-th rows are the first rows, and the
			 	 other rows of $D$ are in reverse order.\\

 				 \If { $D_2=\max\{X_2,Y_2,Z_2\}$ }
 				 	{
 				 	 Let $A^\prime := D $. \\
 				 	 Let $B^\prime$ be the rectangle with the second largest second row over all $(X,Y,Z)$.\\
 				 	 Let $C^\prime$ be the rectangle with the smallest second row over all $(X,Y,Z)$.\\
 				 	 \If{$(A^\prime,B^\prime,C^\prime)>(A,B,C)$}
 				 	 	{\GOTO 3 and continue with the next triple of the {\bf foreach} loop. 				 	 	}
 				 	}
 				}
 			}   	
	  }
	Add $(A,B,C)$ to $L^\prime$.\\
  }
}

\end{algorithm}

\end{document}